\documentclass[12pt]{amsart}
\usepackage{graphicx}
\newtheorem{theorem}{Theorem}[section]

\newtheorem{lemma}[theorem]{Lemma}

\newtheorem{corollary}[theorem]{Corollary}

\theoremstyle{remark}
\newtheorem{remark}[theorem]{Remark}

\newtheorem{conjecture}[theorem]{Conjecture}

\numberwithin{equation}{section}

\begin{document}

\title[
Unknotting number and number of Reidemeister moves needed for unlinking
]{
Unknotting number and number of Reidemeister moves needed for unlinking
}

\author{Chuichiro Hayashi and Miwa Hayashi}

%\date{\today}

\thanks{The first author is partially supported
by Grant-in-Aid for Scientific Research (No. 22540101),
Ministry of Education, Science, Sports and Technology, Japan.}

\begin{abstract}
% Hass and Nowik
%introduced a certain knot diagram invariant $g \circ I_{lk}$
%by using the smoothing operation and the linking number $lk$.
% Using this invariant,
%they gave an example of an infinite sequence of diagrams of the unknot
%which needs quadratic number of Reidemeister moves 
%for being unknotted 
%with respect to the number of crossings.
% In this paper, we 
Using unknotting number, 
we introduce a link diagram invariant of Hass and Nowik type,
which changes at most by $2$ under a Reidemeister move.
 As an application, 
we show that a certain infinite sequence of diagrams 
of the trivial two-component link 
need quadratic number of Reidemeister moves
for being unknotted
with respect to the number of crossings.
 Assuming a certain conjecture 
on unknotting numbers of a certain series of composites of torus knots,
we show that the above diagrams
need quadratic number of Reidemeister moves
for being splitted.
\end{abstract}

\maketitle

\section{Introduction}

 In this paper,
we regard that knot is a link with one component,
and assume that links and link diagrams are oriented,
and link diagrams are in the $2$-sphere.
 A Reidemeister move is a local move of a link diagram
as in Figure \ref{fig:Reid123}.
 An RI (resp. II) move
creates or deletes a monogon face (resp. a bigon face).
 An RII move is called matched or unmatched 
according to the orientations of the edges of the bigon
as shown in Figure \ref{fig:matched}.
 An RIII move is performed on a $3$-gon face,
deleting it and creating a new one.
 Any such move does not change the link type.
 As Alexander and Briggs \cite{AB} and Reidemeister \cite{R} showed,
for any pair of diagrams $D_1$, $D_2$ which represent the same link type,
there is a finite sequence of Reidemeister moves
which deforms $D_1$ to $D_2$.

\begin{figure}[htbp]
\begin{center}
\includegraphics[width=8cm]{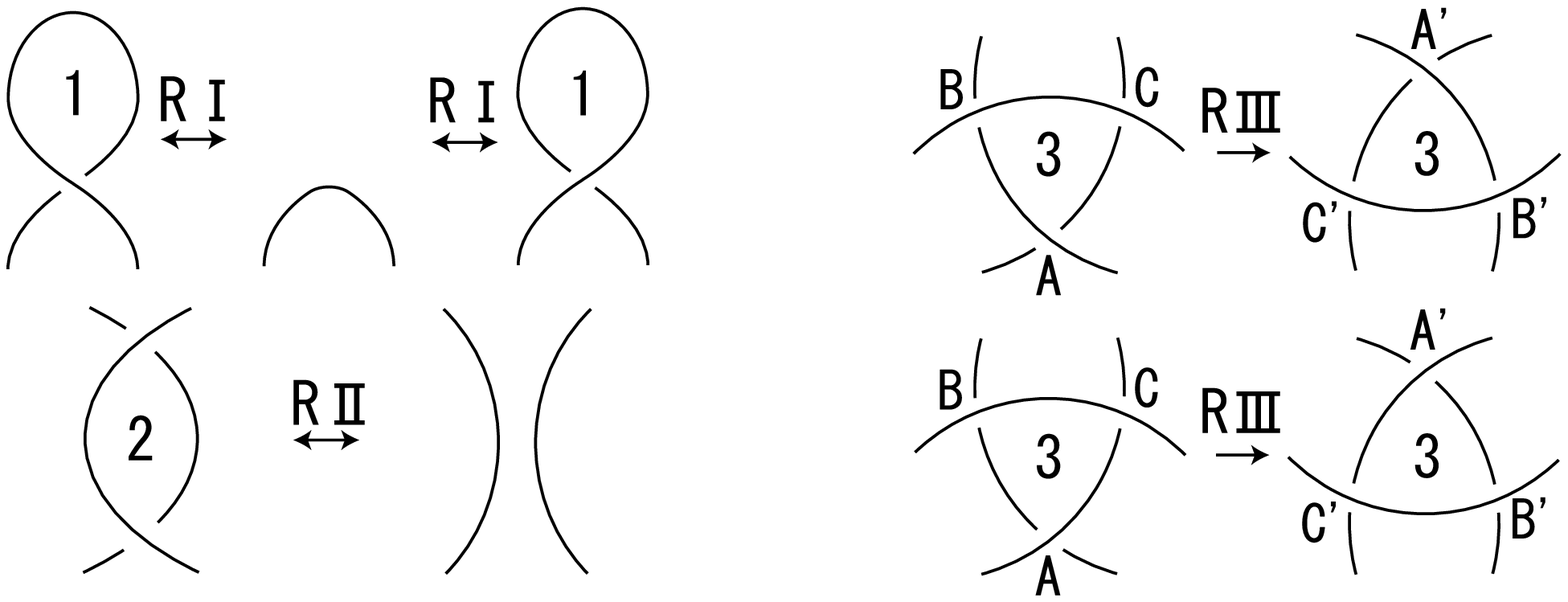}
\end{center}
\caption{}
\label{fig:Reid123}
\end{figure} 

\begin{figure}[htbp]
\begin{center}
\includegraphics[width=5cm]{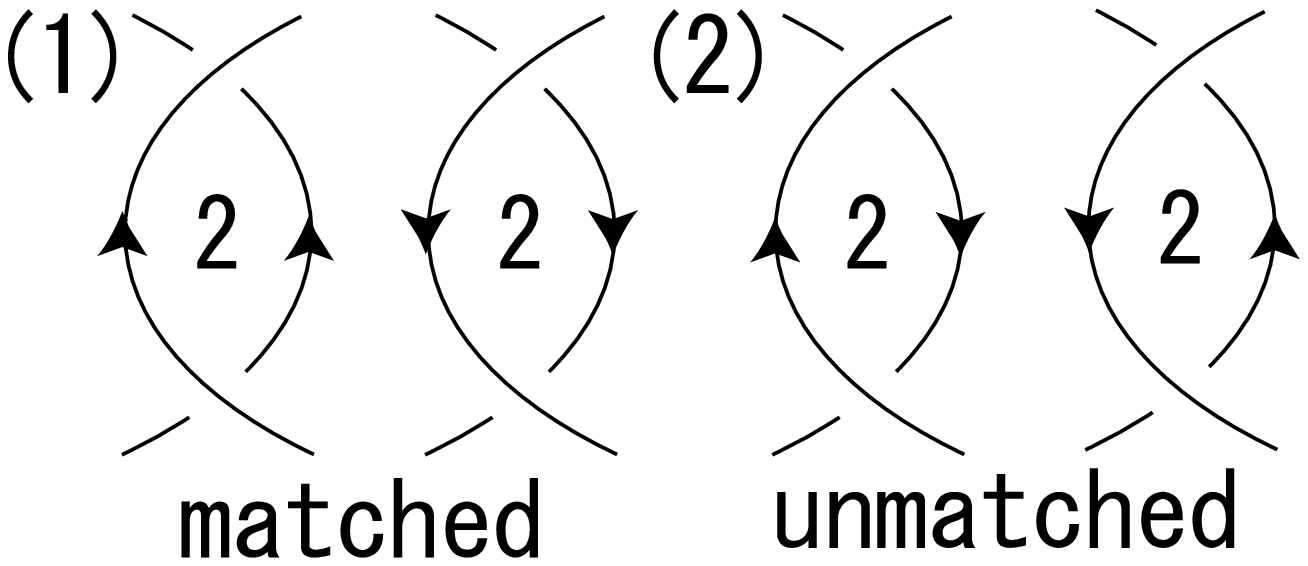}
\end{center}
\caption{}
\label{fig:matched}
\end{figure}

 Necessity of Reidemeister moves of type II and III is studied 
in \cite{O}, \cite{Ma} and \cite{Hg}.
 There are several studies 
of lower bounds for the number of Reidemeister moves
connecting two knot diagrams of the same knot.
 See \cite{H}, \cite{CESS}, \cite{HN1}, \cite{HN2}, \cite{HH}, \cite{HHSY}.
 In particular,
Hass and Nowik introduced in \cite{HN1}
a certain knot diagram invariant $I_{lk}$
by using the smoothing operation and the linking number.
% $lk$.
 Let ${\mathbb G}_{\mathbb Z}$ be 
the free abelian group with basis $\{ X_n , Y_n \}_{n \in {\mathbb Z}}$.
 The invariant $I_{lk}$ 
assigns an element of ${\mathbb G}_{\mathbb Z}$ to a knot diagram.
 In \cite{HN2},
they showed that a certain homomorphism 
$g : {\mathbb G}_{\mathbb Z} \rightarrow {\mathbb Z}$ 
gives a numerical invariant $g(I_{lk})$ of a knot diagram
which changes at most by one under a Reidemeister move.
 They gave an example of an infinite sequence
of diagrams of the trivial knot
such that 
the $n$-th one
has $7n-1$ crossings,
can be unknotted by $2n^2+3n$ Reidemeister moves,
and needs at least $2n^2+3n-2$ Reidemeister moves 
for being unknotted.

 The above papers studied Reidemeister moves on knot diagrams
rather than link diagrams.
 In this paper,
we introduce a link diagram invariant $iu(D)$ of Hass and Nowik type
using unknotting number.
 The invariant $iu(D)$ changes at most by $2$ under a Reidemeister move.
 As an application, 
we show that a certain infinite sequence of diagrams 
of the trivial two-component link 
need quadratic number of Reidemeister moves
for being deformed into a diagram with no crossing.
%unknotted
%with respect to the number of crossings.
 Assuming a certain conjecture 
on unknotting numbers of a certain series of composites of torus knots,
we show that the above diagrams
need quadratic number of Reidemeister moves
for being splitted.

\begin{figure}[htbp]
\begin{center}
\includegraphics[width=4cm]{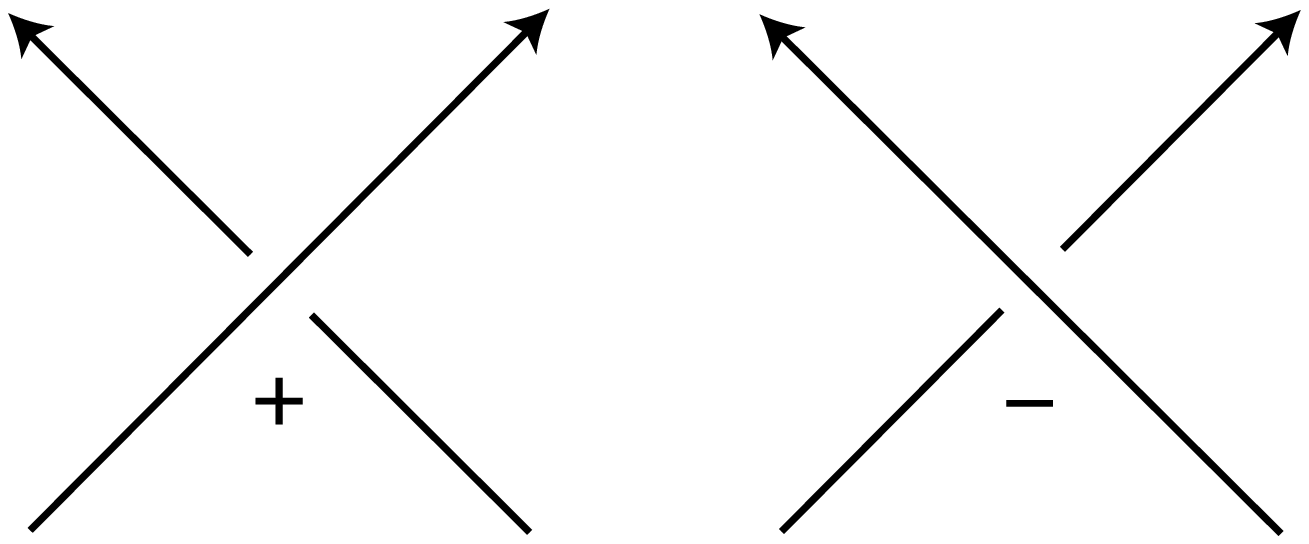}
\end{center}
\caption{}
\label{fig:sign}
\end{figure}

 We roughly sketch the definition of $iu(D)$.
 (Precise descriptions of the definitions of the unknotting number 
and $iu(D)$ are given in Section \ref{section:iu}.)
 For a link $L$ of $m$ components,
the {\it unknotting number} $u(L)$ of $L$ is 
the $X$-Gordian distance 
between $L$ and the trivial link of $m$ components.
 We define the link diagram invariant $iu(D)$ as below.
 Let $D$ be a diagram of an oriented link (possibly a knot) $L_D$.
 We assume that $D$ is in the $2$-sphere.
 For a crossing $p$ of $D$,
let $D_p$ denote the link (possibly a knot) obtained from $D$ 
by performing a smoothing operation at $p$
with respect to the orientation of $D$.
 Note that $D_p$ is a link rather than a diagram.
 If $L_D$ has $m_D$ components,
then $D_p$ has $m_D +1$ components when $p$ is a crossing between subarcs of the same component,
and $m_D -1$ components when $p$ is a crossing between subarcs of distinct component.
 Then we set 
$iu(D) = \sum_{p \in {\mathcal C}(D)} {\rm sign}(p) |\Delta u(D_p)|$,
where ${\mathcal C}(D)$ is the set of all the crossings of $D$,
and $\Delta u(D_p)$ is the difference between the unknotting numbers of $D_p$ and $L_D$, i.e., 
$\Delta u (D_p)= u(D_p) - u(L_D)$.
 The sign of a crossing ${\rm sign}(p)$ is defined as in Figure \ref{fig:sign}
as usual.
 We set $iu(D)=0$ for a diagram $D$ with no crossing.

% Let $w(D)$ denote the writhe of a link diagram $D$, 
%i.e., the sum of the signs of all crossings of $D$.
 When $D$ represents a knot,
$iu(D)+w(D)$ with $w(D)$ being the writhe
is the Hass-Nowik knot diagram invariant $g(I_{\phi} (D))$ 
introduced in \cite{HN1} and \cite{HN2}
with $g$ being the homomorphism with $g(X_k)= |k|+1$ and $g(Y_k) = -|k|-1$ as in \cite{HN2},
and $\phi$ being the difference of the unknotting numbers $\Delta u$.

\begin{theorem}\label{theorem:main}
 The link diagram invariant $iu(D)$ does not change
under an RI move and an unmatched RII move,
and changes at most by one under a matched RII move,
and at most by two under an RIII move.
\end{theorem}

 The above theorem is proved in Section \ref{section:iu}.

\begin{corollary}\label{corollary:main}
 Let $D_1$ and $D_2$ be link diagrams of the same oriented link.
 We need at least $| iu(D_1) - iu(D_2) | /2$ 
matched RII and RIII moves
to deform $D_1$ to $D_2$ by a sequence of Reidemeister moves.
 In particular,
when $D_2$ is a link diagram with no crossing,
we need at leat $|iu (D_1)|/2$ 
matched RII and RIII moves.
\end{corollary}

 Note that, for estimation of the unknotting number,
we can use the signature and the nullity 
(see Theorem 10.1 in \cite{Mu} and Corollary 3.21 in \cite{KT})
or the sum of the absolute values of linking numbers
over all pairs of components.

 For a link diagram $D$, 
the sum of the signs of all the crossings
is called the writhe and denoted by $w(D)$.
 It does not change under an RII or RIII move
but increases (resp. decreases) by $1$ 
under an RI move creating a positive (resp. negative) crossing.
 Set
$iu_{\epsilon, \delta} (D) 
= iu(D) + \epsilon (\dfrac{1}{2}c(D) + \delta \dfrac{3}{2}w(D))$
for a link diagram $D$,
where $\epsilon = \pm 1$, $\delta = \pm 1$
and $c(D)$ denotes the number of crossings of $D$.
 Then we have the next corollary.

\begin{corollary}\label{corollary:main2}
 The link diagram invariant $iu_{\epsilon, +1} (D)$ (resp. $iu_{\epsilon, -1} (D)$)
increases by $2\epsilon$ under an RI move creating a positive (resp. negative) crossing,
decreases by $\epsilon$ under an RI move creating a negative (resp. positive) crossing,
increases by $\epsilon$ under an unmatched RII move, 
changes at most by $2$ under a matched RII move,
and changes at most by $2$ under an RIII move.

 Let $D_1$ and $D_2$ be link diagrams of the same oriented link.
 We need at least $| iu_{\epsilon, \delta}(D_1) - iu_{\epsilon, \delta}(D_2) | /2$ 
Reidemeister moves to deform $D_1$ to $D_2$.
 In particular,
when $D_2$ is a link diagram with no crossing,
we need at least $|iu_{\epsilon, \delta} (D_1)|/2$ Reidemeister moves.
\end{corollary}

\begin{remark}
 We can set 
$iu'(D) = 
\sum_{p \in {\mathcal S}(D)} {\rm sign}(p) |\Delta u(D_p)|
+
\sum_{p \in {\mathcal M}(D)} {\rm sign}(p) \cdot u(D_p)$,
where ${\mathcal S}(D)$ denotes the all crossings of $D$
between subarcs of the same component,
and ${\mathcal M}(D)$ denotes the all crossings of $D$
between subarcs of distinct components.
 Then $iu'(D)$ has the same properties as those of $iu(D)$ 
described in Theorem \ref{theorem:main}, Corollary \ref{corollary:main}
and Corollary \ref{corollary:main2}.
\end{remark}

 Suppose that an $m$-component link $L$ is split,
and there is a splitting $2$-sphere
which separates components $J_1, J_2, \cdots, J_k$ of $L$
form the other components $K_1, K_2, \cdots, K_{\ell}$ of $L$.
 Set $J = \{ J_1, \cdots, J_k \}$,
and $K = \{ K_1, \cdots, K_{\ell} \}$.
 Let $D$ be a diagram of $L$ in the $2$-sphere.
 We denote by ${\mathcal C}(J,K,D)$ 
the set of all crossings of $D$
between a subarc of $J_i$ and a subarc of $K_j$
for some $1 \le i \le k$ and $1 \le j \le \ell$.
 Then we set
$iu'(J,K,D) = \sum_{p \in {\mathcal C}(J,K,D)} {\rm sign}(p) \cdot u(D_p)$.
 If ${\mathcal C}(J,K,D) = \emptyset$,
then we set $iu'(J,K,D) = 0$.

 A similar argument as the proof of Theorem \ref{theorem:main} shows
the next theorem.
 We omit the proof.

\begin{theorem}\label{theorem:main_split}
 Let $L$, $J$, $K$, $D$ be as above.
 The link diagram invariant $iu'(J,K,D)$ does not change
under an RI move and an unmatched RII move,
and changes at most by one under a matched RII move,
and at most by two under an RIII move.
 We need at least $| iu'(J,K,D) | /2$ 
matched RII and RIII moves
to deform $D$ to a split link diagram $E$ 
with ${\mathcal C}(J,K,E) = \emptyset$
by a sequence of Reidemeister moves.
\end{theorem}

\begin{figure}[htbp]
\begin{center}
\includegraphics[width=5cm]{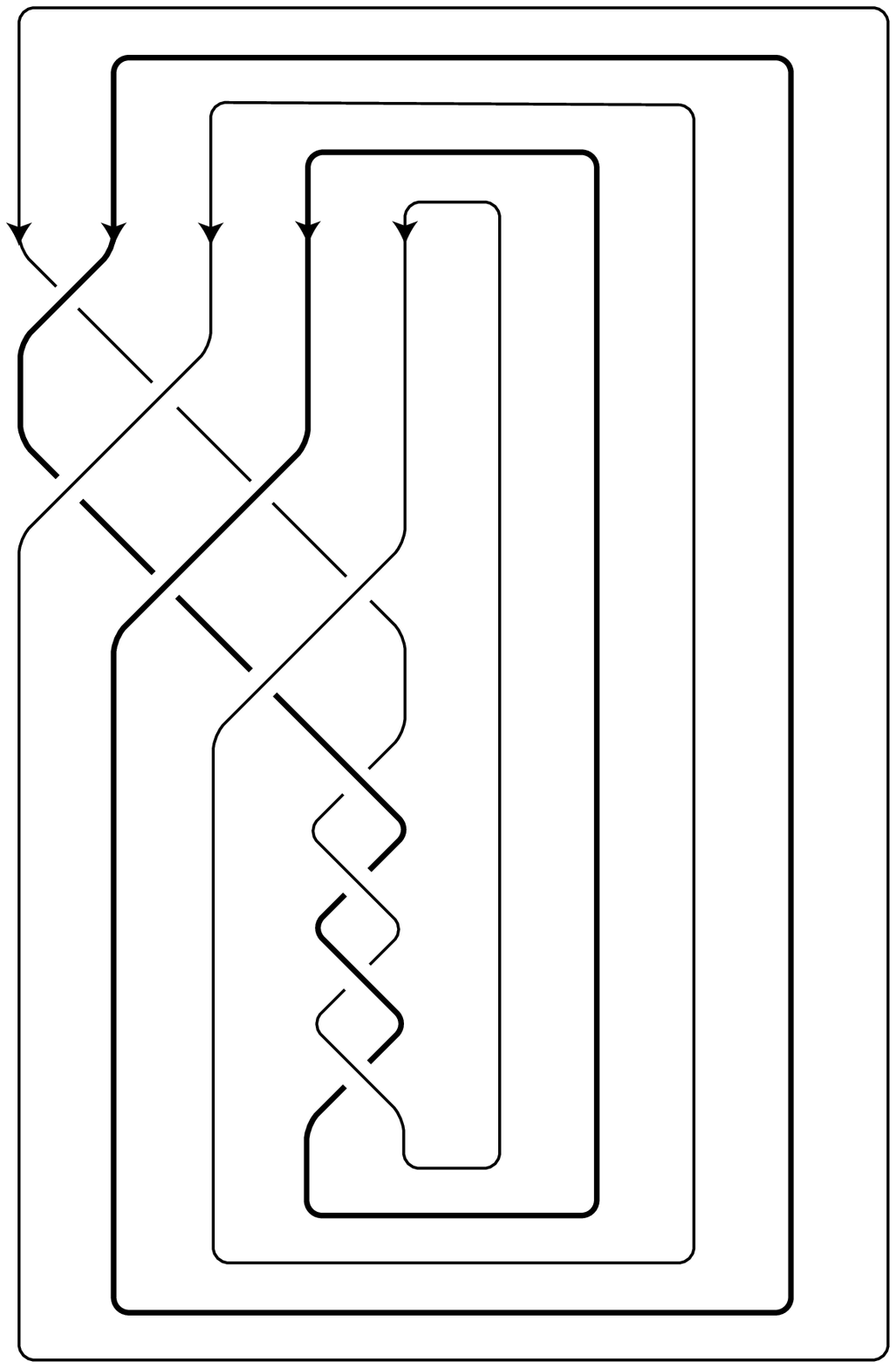}
\end{center}
\caption{}
\label{fig:Dn}
\end{figure}

\begin{figure}[htbp]
\begin{center}
\includegraphics[width=5cm]{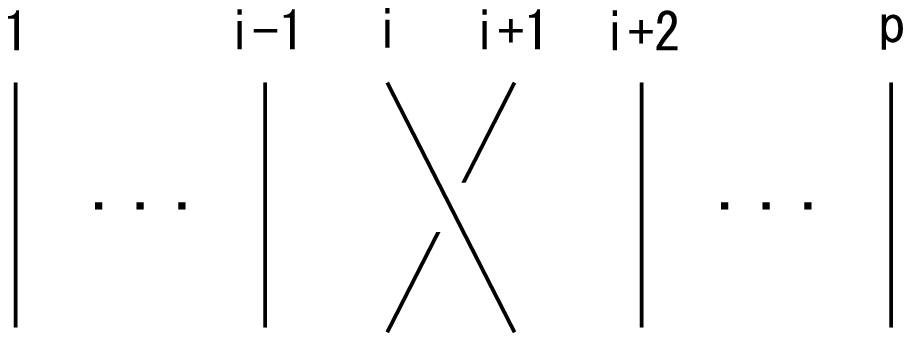}
\end{center}
\caption{}
\label{fig:sigma}
\end{figure}

 Let us describe the link diagram $D_n$ with $n$ being a natural number.
 See Figure \ref{fig:Dn},
where $D_n$ with $n=4$ is depicted.
 For any $i \in \{ 1, 2, \cdots, n-1 \}$,
let $\sigma_i$ be the generator of the $n$-braid group $B_n$,
which denotes the braid
where the $i$-th strand crosses over the $(i+1)$\,st strand
(Figure \ref{fig:sigma}).
% We let $b_i$ denote $\sigma_i^{-1}$ for short.
 Then, $D_n$ is the closure of the $(n+1)$-braid
$\sigma_1^{-1} (\sigma_2^{-1} \sigma_1^{-1})(\sigma_3^{-1} \sigma_2^{-1})
\cdots(\sigma_n^{-1} \sigma_{n-1}^{-1}) \sigma_n^n$.
 We orient $D_n$ so that it is descending on the braid.
 Thus $D_n$ has $2n-1$ positive crossings and $n$ negative crossings.

\begin{theorem}\label{theorem:sequence}
 For any natural number $n$,
the diagram $D_n$ of the trivial two-component link
can be deformed to a diagram with no crososing
by a sequence of $(n^2+3n-2)/2$ Reidemeister moves
which consists of 
$n-1$ RI moves deleting a positive crossing,
$n$ matched RII moves deleting a bigon face,
and $(n-1)n/2$ RIII moves.
 Moreover, 
any sequence of Reidemeister moves 
bringing $D_n$ to a diagram with no crossing
must contain
$\dfrac{1}{2}[3n-2+2\displaystyle\sum_{k=1}^{n-1} u(T(2,k)\ \sharp\ T(2,-k)) 
+ u(T(2,n)\ \sharp\ T(2,-n)) ]$
or larger number of Reidemeister moves,
where $T(2,k)$ is the $(2,k)$-torus link, 
$T(2,-k)$ is the mirror image of $T(2,k)$,
and $\sharp$ denotes the connected sum. 
\end{theorem}

 We estimate the sum 
$\Sigma=\displaystyle\sum_{k=1}^{n-1} u(T(2,k)\ \sharp\ T(2,-k))
+ u(T(2,n)\ \sharp\ T(2,-n))$.
 For an even number $k$,
%the linking number of $T(2,k)$ is $k/2$
%and that of $T(2,-k)$ is $-k/2$.
% Hence the unknotting number of $T(2,k)\ \sharp\ T(2,-k)$ 
%is greater than or equal to $|k/2|+|-k/2|=k$.
% It is easy to see that $k$ $X$-moves unknots $T(2,k)\ \sharp\ T(2,-k)$.
% Thus 
the link $T(2,k)\ \sharp\ T(2,-k)$ has $3$ components.
 By using the sum of the absolute values of the linking numbers,
we can see easily that $u(T(2,k) \sharp\ T(2,-k)) = k$.
 For an odd number $k$ larger than $1$,
$T(2,k)\ \sharp\ T(2,-k)$ is a composite knot.
 A composite knot has $2$ or greater unknotting number,
which was shown in \cite{S} by M. Scharlemann.
 Hence we have $\Sigma \ge (n^2+4n-8)/2$ when $n$ is even,
and $\Sigma \ge (n^2+4n-9)/2$ when $n$ is odd.
 If the conjecture on the unknotting number below holds,
then $\Sigma = n^2 - n$.

\begin{conjecture}\label{conjecture:unknotting}
 $u(T(2,k)\ \sharp\ T(2,-k)) = k-1$ for any odd integer $k$.
\end{conjecture}

 Applying Theorem \ref{theorem:main_split} to the diagram $D_n$,
we have the next theorem.
 We omit the proof.

\begin{theorem}\label{theorem:sequence_split}
 Any sequence of Reidemeister moves 
bringing $D_n$ to a disconnected diagram
must contain
$\displaystyle\sum_{k=1}^{\frac{n}{2}-1} u(T(2,2k+1)\ \sharp\ T(2,-(2k+1)))$
or larger number of Reidemeister moves
when $n$ is even, 
and 
$\dfrac{1}{2}[
2\displaystyle\sum_{k=1}^{\frac{n-1}{2}-1} u(T(2,2k+1)\ \sharp\ T(2,-(2k+1))) 
+ u(T(2,n)\ \sharp\ T(2,-n)) ]$
or larger number of Reidemeister moves
when $n$ is odd.
\end{theorem}

 Since the unknotting number of a composite knot
is greater than or equal to $2$,
the above number is larger than or equal to $n-2$. 
 If Conjecture \ref{conjecture:unknotting} is true,
then the above number is equal to $(n^2-2n)/4$ when $n$ is even,
and to $(n^2-2n+1)/4$ when $n$ is odd.

% The estimation for the number of RI moves is easily obtained
%by using writhe.
 The precise definition of link diagram invariant $iu(D)$ is given 
in Section \ref{section:iu},
where changes of the value of the invariant under Reidemeister moves
are studied.
 The sequence of Reidemeister moves in Theorem \ref{theorem:sequence}
is described in Section \ref{section:deformation}.
 In Section \ref{section:calculation}, 
we calculate $iu(D_n)$, to prove Theorem \ref{theorem:sequence}.
% The proof of Theorem \ref{theorem:sequence} 
%is given in Section \ref{section:ProofSequence}.

%----------------------------------------------------------------------

\section{link diagram invariant}\label{section:iu}

 A link is called the {\it trivial $n$-component link}
if it has $n$ components and bounds a disjoint union of $n$ disks.
 The trivial $n$-component link admits
a {\it trivial diagram} with no crossing.

 Let $L$ be a link with $n$ components,
and $D$ a diagram of $L$.
 We call a sequence of Reidemeister moves and crossing changes on $D$
an {\it $X$-unknotting sequence} in this paragraph,
if it deforms $D$ into a (possibly non-trivial) diagram 
of the trivial $n$-component link.
 The {\it length} of an $X$-unknotting sequence
is the number of crossing changes in it.
 The minimum length among all the $X$-uknotting sequences on $D$
is called the {\it uknotting number} of $L$.
 We denote it by $u(L)$.
 Clearly, $u(L)$ depends only on $L$ and not on $D$.

\begin{figure}[htbp]
\begin{center}
\includegraphics[width=4cm]{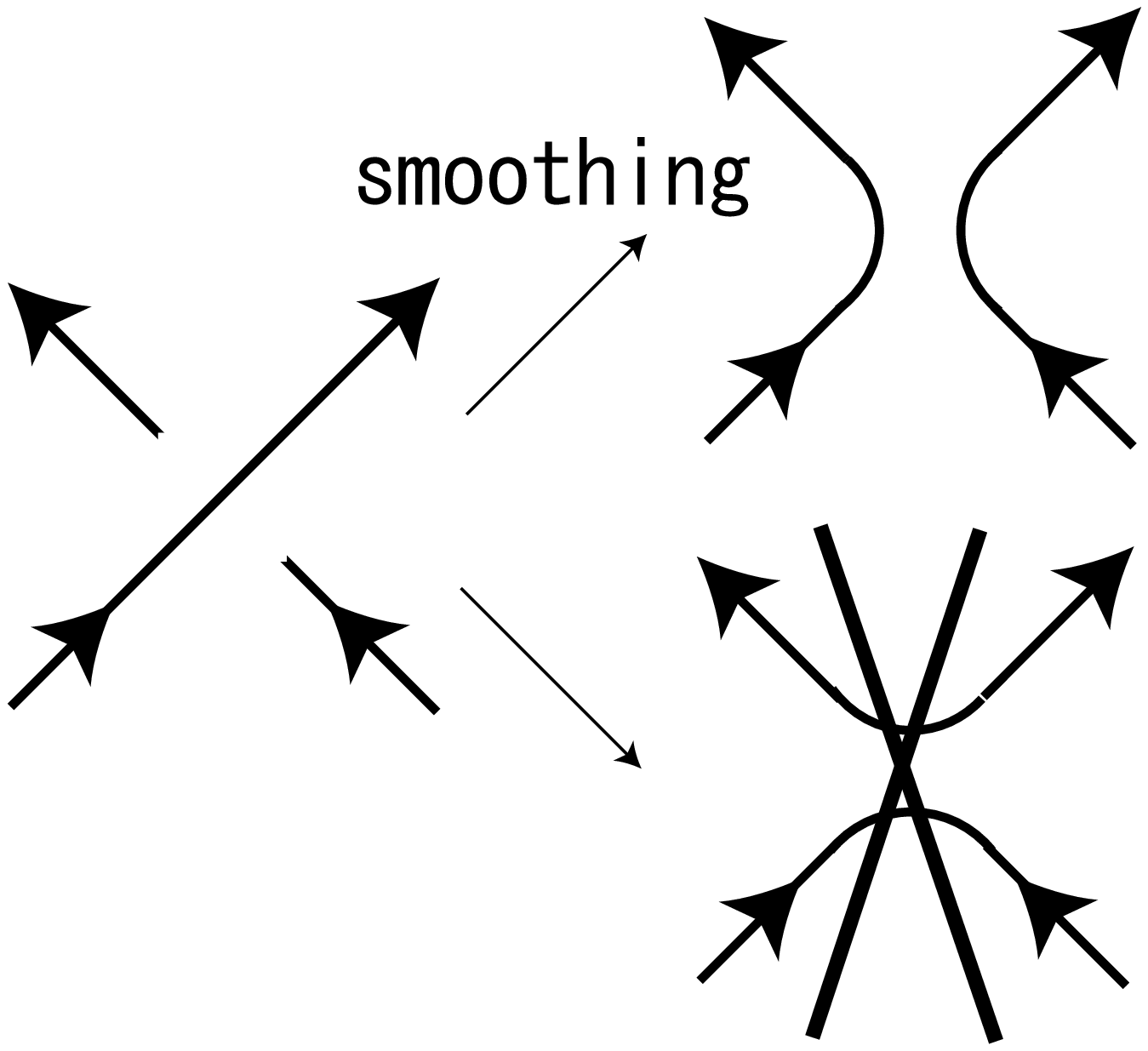}
\end{center}
\caption{}
\label{fig:smoothing}
\end{figure}

 Let $D$ be an oriented link diagram in the $2$-sphere,
$L_D$ the link represented by $D$, 
$p$ a crossing of $D$,
and $D_p$ the link (rather than a diagram) obtained from $D$
by performing the smoothing operation on $D$ at $p$ as below.
 We first cut the link at the two preimage points of $p$.
 Then we obtain the four endpionts.
 We paste the four short subarcs of the link near the endpoints
in the way other than the original one
so that their orientations are connected consistently.
% For an example of a smoothing operation,
 See Figure \ref{fig:smoothing}.

 We set $iu(D)$ to be 
the sum of 
the absolute value of the difference of the unknotting numbers 
$\Delta u(D_p) = u(D_p)-u(L_D)$ with the sign of $p$
over all the crossings of $D$, i.e., 
$$iu(D) = \sum_{p\in {\mathcal C}(D)} {\rm  sign}(p) \cdot |\Delta u(D_p)|$$
where ${\mathcal C}(D)$ is the set of all the crossings of $D$.
 For a diagram $D$ with no crossing, we set $iu(D)=0$.
 Note that, when $D$ represents a knot,
$iu(D)+w(D)$ with $w$ being the writhe 
is equal to Hass-Nowik invariant $g(I_{\phi}(D))$
with $\phi$ being $\Delta u$
and $g$ being the homomorphism 
with $g(X_k) = |k|+1$ and $g(Y_k) = -|k|-1$ (see \cite{HN1} and \cite{HN2}).
%with $g(X_{u(D_p)}) = u(D_p)$ and $g(Y_{u(D_p)}) = -u(D_p)$.

\begin{theorem}
{ 
$iu(D)$ does not change under an RI move and an unmatched RII move,
changes at most by one under a matched RII move,
and at most by two under an RIII move.
}
\end{theorem}

%\begin{figure}[htbp]
%\begin{center}
%\includegraphics[width=10cm]{s-trigon.eps}
%\end{center}
%\caption{}
%\label{fig:trigon}
%\end{figure}

\begin{proof}
 The proof is very similar to the arguments in Section 2 in \cite{HN1}.
 Let $D, E$ be link diagrams
such that $E$ is obtained from $D$ by a Reidemeister move.
 Let $L_D, L_E$ be the links represented by $D, E$ respectively.
 Note that $L_D$ and $L_E$ are the same link.

 First, 
we suppose that $E$ is obtained from $D$ by an RI move creating a crossing $a$.
 Then the link $E_a$ differs from $L_E$ by an isolated single trivial component,
and hence $u(E_a) = u(L_E)$.
 Then the contribution of $a$ to $iu$ is $\pm(u(E_a)-u(L_E))= 0$.
 The contribution of any other crossing $x$ to $iu$ is unchanged
since the RI move shows 
that $L_D$ and $L_E$ are the same link
and so $D_x$ and $E_x$ are.
 Thus an RI move does not change $iu$, i.e., $iu(D)=iu(E)$.

 When $E$ is obtained from $D$ by an RII move creating a bigon face,
let $x$ and $y$ be 
the positive and negative crossings at the corners of the bigon.
 If the RII move is unmatched,
then $E_x$ and $E_y$ are the same link.
 Hence $u(E_x)=u(E_y)$ and 
$|iu(E)-iu(D)| = ||u(E_x)-u(L_E)| - |u(E_y)-u(L_E)|| = 0$.
 If the RII move is matched,
then $E_x$ and $E_y$ differ by a crossing change,
and hence their unknotting numbers differ by at most one,
i.e., $|u(E_x)-u(E_y)| \le 1$. 
 Hence 
$|iu(E)-iu(D)| 
= ||u(E_x)-u(L_E)| - |u(E_y)-u(L_E)||
\le |(u(E_x)-u(L_E))-(u(E_y)-u(L_E))|
= |u(E_x)-u(E_y)|
\le 1$.

 We consider the case where $E$ is obtained from $D$ by an RIII move.
 For the crossing $x$ 
between the top and the middle strands of the trigonal face
where the RIII move is applied,
$D_x$ and $E_x$ are the same link.
 Hence the contribution of $x$ to $iu$ is unchanged.
 The same is true for the crossing $y$ between the bottom and the middle strands.
 Let $z$ be the crossing between the top and the bottom strands.
 Then $D_z$ and $E_z$ differ by two crossing changes and Redemeister moves,
and hence $|u(E_z)-u(D_z)| \le 2$. 
% See Figure \ref{fig:trigon}.
 Thus
$|iu(E)-iu(D)|
= ||u(E_z)-u(L_E)| - |u(D_z)-u(L_D)||
\le |(u(E_z)-u(L_E))-(u(D_z)-u(L_D))|
= |u(E_z)-u(D_z)|
\le 2$.
\end{proof}

%-----------------------------------------------------------------------

\section{Unknotting sequence of Reidemeister moves on $D_n$}\label{section:deformation}

%\begin{figure}[htbp]
% \begin{minipage} {0.4\hsize}
%  \begin{center}
%   \includegraphics[width=60mm]{deform1.eps}
%  \end{center}
%  \caption{}
%  \label{fig:deform1}
% \end{minipage}
% \begin{minipage} {0.4\hsize}
%  \begin{center}
%   \includegraphics[width=60mm]{deform2.eps}
%  \end{center}
%  \caption{}
%  \label{fig:deform2}
% \end{minipage} 
%\end{figure}

 In this section,
we deform the link diagram $D_n$ to a diagram with no crossing
by a sequence of Reidemeister moves.

\begin{lemma}\label{lemma:deformation1}
% The closure of the braid 
%$b=\sigma_1^{-k} \sigma_2^{-1} \cdots \sigma_{m-1}^{-1}
%\sigma_1^{-1} \sigma_2^{-1} \cdots \sigma_{m-1}^{-\ell}$
%can be deformed into 
%that of
%$b'=\sigma_1^{-(k+1)} \sigma_2^{-1} \cdots \sigma_{m-2}^{-1}
%\sigma_1^{-1} \sigma_2^{-1} \cdots \sigma_{m-2}^{-\ell}$
%by a sequence of $k$ RIII moves and a single RI move.
 The closure of the braid 
$b=
\sigma_1^{-k}
(\sigma_2^{-1} \sigma_1^{-1}) 
(\sigma_3^{-1} \sigma_2^{-1})
\cdots 
(\sigma_m^{-1} \sigma_{m-1}^{-1}) 
\sigma_m^{\ell}$
can be deformed into 
that of
$b'=
\sigma_1^{-k-1}
(\sigma_2^{-1} \sigma_1^{-1}) 
(\sigma_3^{-1} \sigma_2^{-1})
\cdots 
(\sigma_{m-1}^{-1} \sigma_{m-2}^{-1}) 
\sigma_{m-1}^{\ell}$
by a sequence of $k$ RIII moves and a single RI move.
\end{lemma}

\begin{figure}[htbp]
\begin{center}
\includegraphics[width=8cm]{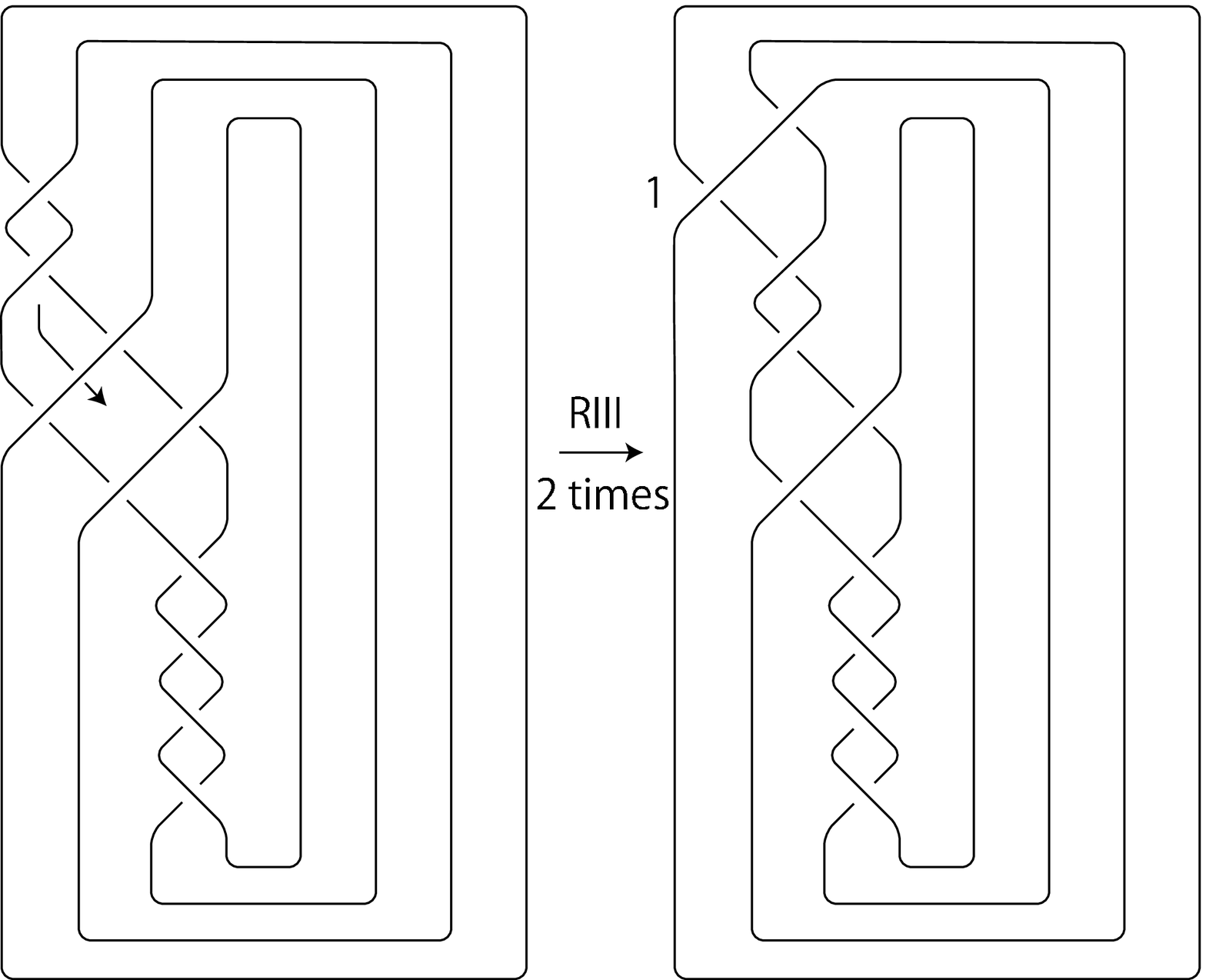}
\end{center}
\caption{}
\label{fig:deformation}
\end{figure}

\begin{proof}
% Far commutativity allows us 
%to move the crossing $\sigma_1^{-1}$ at the middle
%and deform the closed braid $b$
%to that of 
%\newline
%$\sigma_1^{-k} \sigma_2^{-1} \sigma_1^{-1} 
%\sigma_3^{-1} \cdots \sigma_{m-1}^{-1}
%\sigma_2^{-1} \cdots \sigma_{m-1}^{-\ell}$.
% This deformation does not change the link diagram.
% Then, 
 Applying the braid relation 
$\sigma_1^{-1} \sigma_2^{-1} \sigma_1^{-1}
=\sigma_2^{-1} \sigma_1^{-1} \sigma_2^{-1}$
repeatedly $k$ times,
we obtain the closed braid of 
$\sigma_2^{-1} \sigma_1^{-1} \sigma_2^{-k}
(\sigma_3^{-1} \sigma_2^{-1})
\cdots 
(\sigma_m^{-1} \sigma_{m-1}^{-1}) 
\sigma_m^{\ell}$.
 This is accomplished by a sequence of $k$ RIII moves.
 See Figure \ref{fig:deformation}.
 Thus we can apply RI move (Markov's destabilization)
on the outermost region which is a monogon face,
to obtain $b'$.
 Note that this remove the only $\sigma_1^{-1}$,
and reduces the suffix numbers of $\sigma_i$ ($i \ge 2)$ by one.
\end{proof}

% A similar argument proves the next lemma. 
% We omit the proof.
% The crossing $\sigma_{m-1}^{-1}$ at the middle
%is moved to the bottom of the braid.
%
%\begin{lemma}\label{lemma:deformation2}
% The closure of the braid 
%$\sigma_1^{-k} \sigma_2^{-1} \cdots \sigma_{m-1}^{-1}
%\sigma_1^{-1} \sigma_2^{-1} \cdots \sigma_{m-1}^{-\ell}$
%can be deformed into 
%that of
%$\sigma_1^{-k} \sigma_2^{-1} \cdots \sigma_{m-2}^{-1}
%\sigma_1^{-1} \sigma_2^{-1} \cdots \sigma_{m-2}^{-(\ell+1)}$
%by a sequence of $\ell$ RIII moves and a single RI move.
%\end{lemma}

 We can apply the deformation in Lemma \ref{lemma:deformation1}
%and Lemma \ref{lemma:deformation2} alternately
repeatedly $n-1$ times
to deform $D_n$ into the closure of the $2$-braid 
$\sigma_1^{-n} \sigma_1^n$.
 Then a sequence of $n$ matching RII moves
deletes all the crossings.
% In the case of $n=6$, the deformation is described 
%in Figures ???.
 If $n$ is a natural number,
then this deformation consists of $n-1$ RI moves deleting a positive crossing,
$n$ matched RII moves deleting a bigon
and $1+2+\cdots + (n-1) = (n-1)n/2$ RIII moves.

 Thus the former half of Theorem \ref{theorem:sequence} holds.

%----------------------------------------------------------------------

\section{Calculation of $iu(D_n)$ and proof of Theorem \ref{theorem:sequence}}
\label{section:calculation}

\begin{lemma}\label{lemma:iuD2n}
$iu(D_n)=
2\displaystyle\sum_{k=1}^{n-1} u(T(2,k)\ \sharp\ T(2,-k)) 
+ u(T(2,n)\ \sharp\ T(2,-n))$,
\newline
and $iu_{+1,+1}(D_n)=
3n-2+2\displaystyle\sum_{k=1}^{n-1} u(T(2,k)\ \sharp\ T(2,-k)) 
+ u(T(2,n)\ \sharp\ T(2,-n))$.
\end{lemma}

\begin{figure}[htbp]
\begin{center}
\includegraphics[width=9cm]{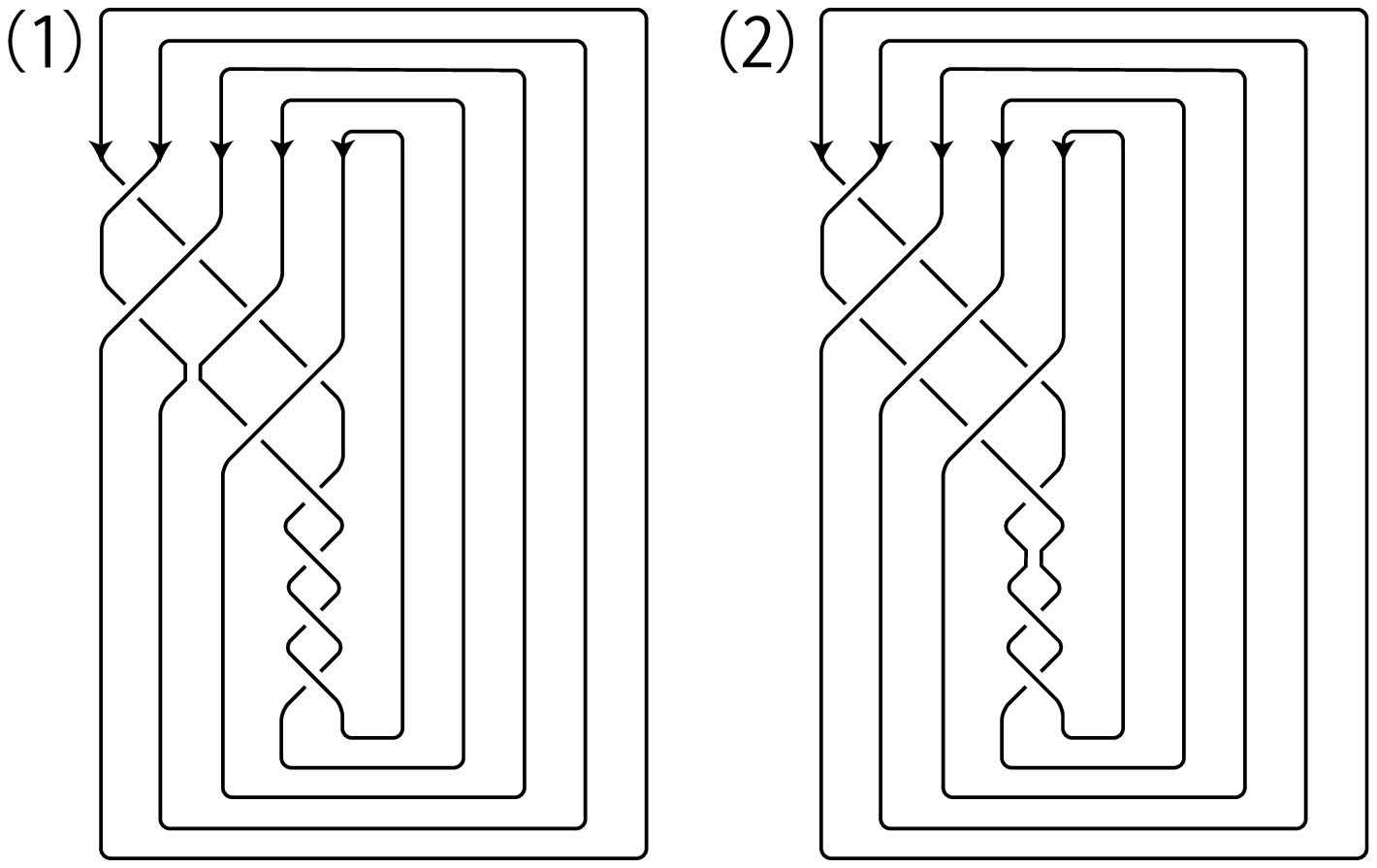}
\end{center}
\caption{}
\label{fig:Dn+}
\end{figure}

\begin{proof}
 If we perform a smoothing at a crossing of $D_n$ 
corresponding to $\sigma_k^{-1}$ for $1 \le k \le n$,
then we obtain the link $T(2,k)\ \sharp\ T(2,-k)$.
 See (1) in Figure \ref{fig:Dn+}.
 A smoothing operation at a crossing of $D_n$
corresponding to $\sigma_n$
yields the trivial knot.
 See (2) in Figure \ref{fig:Dn+}.
 Since $D_n$ represents the trivial $2$-component link,
$\Delta u (D_x) = u (D_x)$ for any crossing $x$ of $D_n$.
 Hence we obtain the above formula of $iu(D_n)$.
 Note that there are two crossing points of $D_n$
corresponding to $\sigma_k^{-1}$ with $1 \le k \le n-1$,
while there is only one crossing point of $D_n$
corresponding to $\sigma_n^{-1}$.

 Since $D_n$ has $2n-1$ positive crossings and $n$ negative crossings,
$c(D_n)/2 + 3w(D_n)/2 = ((2n-1)+n)/2 + 3((2n-1)-n)/2 = 3n-2$.
  Thus we obtain the above formula of $iu_{+1,+1} (D_n)$. 
\end{proof}

\begin{proof}[Proof of Theorem \ref{theorem:sequence}]
 The former half of Theorem \ref{theorem:sequence} is already shown
in Section \ref{section:deformation}.
 The above formula of $iu_{+1,+1} (D_n)$ and Corollary \ref{corollary:main2}
together show the latter half of Theorem \ref{theorem:sequence}.
\end{proof}

%----------------------------------------------------------------

\bibliographystyle{amsplain}

\medskip

\noindent
Chuichiro Hayashi: 
Department of Mathematical and Physical Sciences,
Faculty of Science, Japan Women's University,
2-8-1 Mejirodai, Bunkyo-ku, Tokyo, 112-8681, Japan.
hayashic@fc.jwu.ac.jp

\vspace{3mm}
\noindent
Miwa Hayashi:
Department of Mathematical and Physical Sciences,
Faculty of Science, Japan Women's University,
2-8-1 Mejirodai, Bunkyo-ku, Tokyo, 112-8681, Japan.
\newline
miwakura@fc.jwu.ac.jp

\end{document}